\documentclass[a4paper,11pt]{article}

\usepackage{graphicx}
\usepackage{geometry}
\usepackage{amsthm}
\usepackage{amssymb}
\usepackage{amsmath}
\usepackage{hyperref}
\usepackage{xcolor}
\usepackage{comment}
\usepackage{enumerate}
\usepackage{listings}
\usepackage{complexity}
\usepackage{blkarray}
\usepackage{lmodern}
\usepackage[english]{babel}
\usepackage{accents}
\usepackage[font=small,labelfont=bf]{caption}
\usepackage[font=small,labelfont=normalfont,labelformat=simple]{subcaption}
\graphicspath{{figs/}}

\theoremstyle{plain}
\newtheorem{theorem}{Theorem}
\newtheorem{lemma}{Lemma}[section]
\newtheorem{claim}[lemma]{Claim}
\newtheorem{proposition}[lemma]{Proposition}

\newtheorem{conjecture}[lemma]{Conjecture}
\newtheorem{problem}[lemma]{Problem}

\newtheorem{remark}[lemma]{Remark}
\newtheorem{definition}[lemma]{Definition}

\newcommand{\bivec}[1]{\accentset{\leftrightarrow}{#1}}

\author
{
Raphael Steiner \thanks{Institute of Mathematics, Technische Universit\"at Berlin, Germany} \thanks{Email: \texttt{steiner@math.tu-berlin.de}. Funded by DFG-GRK 2434 Facets of Complexity.}
}

\date{\today}

\title{On coloring digraphs with forbidden induced subgraphs}

\begin{document}
\maketitle

\begin{abstract}
We prove a conjecture by Aboulker, Charbit and Naserasr~\cite{aboulker} by showing that every oriented graph in which the out-neighborhood of every vertex induces a transitive tournament can be partitioned into two acyclic induced subdigraphs. We prove multiple extensions of this result to larger classes of digraphs defined by a finite list of forbidden induced subdigraphs. 
We thereby resolve several special cases of an extension of the famous Gy\'{a}rf\'{a}s-Sumner conjecture to directed graphs stated in~\cite{aboulker}.
\end{abstract}

\section{Introduction}
\paragraph{Notation.}
All graphs and digraphs considered in this paper are simple, that is, they are loopless, between two vertices in a graph there is at most one connecting edge, and between two vertices in a digraph there is at most one arc in each direction. We say that a digraph is an \emph{oriented graph} if it does not contain directed cycles of length two (digons). Given a digraph $D$, we denote by $V(D)$ its vertex-set and by $A(D) \subseteq V(D) \times V(D)$ its set of arcs. We put $v(D):=|V(D)|, a(D):=|A(D)|$. Arcs are denoted as $(u,v)$, where $u$ is the tail of the arc and $v$ is its head. For $v \in V(D)$ we denote by $N_D^+(u), N_D^-(u)$ the sets of out- and in-neighbors of $v$ in $D$, respectively. We drop the subscript $D$ if its is clear from context. We generalize this notation to vertex subsets by putting $N_D^+(X):=\bigcup_{x \in X}{N_D^+(x)}\setminus X$ and $N_D^-(X):=\bigcup_{x \in X}{N_D^-(x)}\setminus X$ for all $X \subseteq V(D)$.  We further denote by $D[X]$ the subdigraph with vertex set $X$ and arc-set $(X \times X) \cap A(D)$. Any digraph of the form $D[X]$ with $\emptyset \neq X \subseteq V(D)$ is called an \emph{induced subdigraph} of $D$. Given a set $X$ of vertices or arcs, we denote by $D-X$ the digraph obtained from $D$ by deleting the elements in $X$. 

For a graph $G$, we denote by $\bivec{G}$ the directed graph with $V(\bivec{G}):=V(G)$ and $A(\bivec{G}):=\{(x,y),(y,x)|xy \in E(G)\}$ and call it the \emph{biorientation} of $G$. For an integer $k \ge 1$ we denote by $\bivec{K}_k$ the biorientation of $K_k$ and call it the \emph{complete digraph} of order $k$, by $\vec{K}_k$ we denote the \emph{transitive tournament} of order $k$, and by $\overline{K}_k$ we denote the digraph consisting of $k$ isolated vertices. By $S_k^+, S_k^-$ we denote the orientation of the star with $k$ leaves in which all arcs are oriented outwards (inwards). By $W_k^+$ and $W_k^-$ we denote the oriented wheel graphs obtained by connecting the leaves of $S_k^+$ and $S_k^-$, respectively, by a directed cycle.
\medskip 

A classical research topic in the theory of graph coloring is to study the chromatic number of graph classes defined by forbidden \emph{induced} subgraphs. Maybe the most famous open problem in this area is the \emph{Gy\'{a}rf\'{a}s-Sumner Conjecture}, which states the following.

\begin{conjecture}[Gy\'{a}rf\'{a}s~\cite{gyarfas} 1975, Sumner~\cite{sumner} 1981]
If $F$ is a forest and $k \in \mathbb{N}$, then there exists an integer $c(F,k) \ge 1$ such that every graph $G$ excluding $F$ and $K_k$ as induced subgraphs satisfies $\chi(G) \le c(F,k)$.
\end{conjecture}

Note that the result claimed by the Gy\'{a}rf\'{a}s-Sumner-conjecture would be best-possible in the following strong sense: Let $\mathcal{F}$ be finite set of graphs, such that the set of graphs $\text{Forb}_\text{ind}(F)$ excluding each member of $\mathcal{F}$ as an induced subgraph has bounded chromatic number. Then at least one member of $\mathcal{F}$ must be a complete graph (for otherwise all complete graphs would be contained in $\text{Forb}_\text{ind}(F)$, and hence the chromatic number of graphs in this class would be unbounded). Similarly, at least one member of $\mathcal{F}$ must be a forest: If every graph in $\mathcal{F}$ would include a cycle, then let $g$ be the maximum length of a cycle appearing in a member of $\mathcal{F}$. It is now clear that every graph whose girth exceeds $g$ is contained in $\text{Forb}_\text{ind}(F)$, but such graphs may have arbitrarily large chromatic number due to a classical result of Erd\H{o}s~\cite{erdos}. Hence, the Gy\'{a}rf\'{a}s-Sumner conjecture may be restated as follows: If $F$ is a finite list of graphs, then $\text{Forb}_\text{ind}(F)$ has bounded chromatic number if and only if $F$ contains a clique and a forest.

Despite being quite popular, the Gy\'{a}rf\'{a}s-Sumner conjecture has not yet been resolved in full generality. Some special cases for which the conjecture has been proved are when $F$ is a subdivision of a star~\cite{scott}, a tree of radius at most two~\cite{kierstead}, a certain kind of caterpillar~\cite{chudnovsky1} or a certain type of a so-called (multi-)broom~\cite{scott2}. 

A widely known generalization of the chromatic number to directed graphs is the \emph{dichromatic number}, which was introduced around 1980 by Erd\H{o}s~\cite{erdos2} and Neumann-Lara~\cite{neumannlara}. Given a digraph $D$, an \emph{acyclic $k$-coloring} of $D$ is an assignment $c:V(D) \rightarrow S$ of colors from a finite color set $S$ of size $k$ to the vertices such that every \emph{color class} $c^{-1}(i), i \in S$ induces an \emph{acyclic} subdigraph of $D$. In other words, an acyclic coloring is a vertex-coloring avoiding monochromatic directed cycles. The \emph{dichromatic number} $\vec{\chi}(D)$ is the smallest integer $k$ for which an acyclic $k$-coloring of $D$ exists. Since around 2000, the dichromatic number has been quite popular in graph theoretical research, and many results have established that is shares interesting structural properties with the chromatic number, see~\cite{aboulker2, aboulker3, andres, bangjensen, heros, bokal, harutyunyan1, harutyunyan2, harutyunyan3, hochstattler, mohar, moharwu} for only a small fraction of recent results on the topic. 

In the same spirit, Aboulker, Charbit and Naserasr~\cite{aboulker} recently initiated the systematic study of the relation between excluded induced subdigraphs and the dichromatic number and asked the following intriguing question.

\begin{problem}\label{prob:excludeinduceddir}
Let $\mathcal{F}$ be a finite set of digraphs. Under which circumstances does there exist $C \in \mathbb{N}$ such that every digraph $D$ without an induced subdigraph isomorphic to a member of $\mathcal{F}$ satisfies $\vec{\chi}(D) \le C$?
\end{problem}

Following the terminology introduced by Aboulker et al.~\cite{aboulker}, we denote by $\text{Forb}_\text{ind}(\mathcal{F})$ the set of digraphs containing no induced subdigraph isomorphic to a member of $\mathcal{F}$. We say that $\mathcal{F}$ is \emph{heroic} if the digraphs in $\text{Forb}_\text{ind}(\mathcal{F})$ have bounded dichromatic number and in this case we denote $\vec{\chi}(\text{Forb}_{\text{ind}}(\mathcal{F})):=\max\{\vec{\chi}(D)|D \in \text{Forb}_{\text{ind}}(\mathcal{F})\}$.

Just as in the undirected case, Aboulker et al.~\cite{aboulker} observed several necessary conditions for a finite set $\mathcal{F}$ of digraphs to be heroic, which we summarize in the following. 

\begin{proposition}[cf.~\cite{aboulker}]\label{prop:necess}
Let $\mathcal{F}$ be a finite heroic set of digraphs. Then $\mathcal{F}$ must contain
\begin{itemize}
\item a complete digraph $\bivec{K}_k$ for some $k \in \mathbb{N}$,
\item a biorientation of a forest,
\item an orientation of a forest,
\item a tournament, i.e., an orientation of a complete graph.
\end{itemize}
\end{proposition}

Inspired by yet another important conjecture in graph theory, the \emph{Erd\H{o}s-Hajnal-Conjecture}, in~\cite{heros} Berger, Choromanski, Chudnovsky, Fox, Loebl, Scott, Seymour and Thomass\'{e} studied the dichromatic number of tournaments which exclude a \emph{single} fixed tournament $H$ as a(n induced) subdigraph. In this paper, the authors defined a \emph{hero} as a tournament $H$ such that the tournaments exluding isomorphic copies of $H$ have bounded dichromatic number. In other words, a digraph $H$ is a hero if the set $\{\bivec{K}_2,\overline{K}_2,H\}$ is heroic. The main result of Berger et al.~in~\cite{heros} is a recursive characterization of heroes.
It follows directly from this characterization that every transitive tournament and all tournaments on at most four vertices are heroes. 

It is a natural aim to characterize the finite heroic sets $\mathcal{F}$ of digraphs similar to what is claimed by the Gy\'{a}rf\'{a}s-Sumner-Conjecture for undirected graphs. In contrast to undirected graphs, only heroic sets of size at least $3$ are interesting to consider, as the necessary conditions from Proposition~\ref{prop:necess} directly imply that $\{\bivec{K}_2,\vec{K}_2\}$ is the only heroic set of size two (and is so trivially). Aboulker et al.~in~\cite{aboulker} proved that every heroic set of size three must be of one of the following types:

\begin{itemize}
\item $\{\vec{K}_2,\bivec{F},\bivec{K}_k\}$ for a forest $F$ and a number $k \in \mathbb{N}$,
\item $\{\bivec{K}_k,\overline{K}_\alpha,H\}$ for $k, \alpha \in \mathbb{N}$ and a hero $H$ such that $k=2$ or $H$ is transitive, or
\item $\{\bivec{K}_2, F, H\}$ for some oriented star forest\footnote{An oriented star forest is a disjoint union of orientations of stars.} $F$ and a hero $H$, or
\item $\{\bivec{K}_2,F,\vec{K}_k\}$ for some oriented forest $F$ and some $k \in \mathbb{N}$.
\end{itemize}
They then ventured to propose the conjecture that every one of the above triples is indeed heroic, thus claiming a complete description of the heroic sets of size $3$. 

Note that since $\vec{K}_2$-free digraphs amount exactly to the biorientations of undirected graphs, and since dichromatic number and chromatic number coincide on these, the conjecture of Aboulker et al.~corresponds exactly to the undirected Gy\'{a}rf\'{a}s-Sumner-Conjecture when restricting to the triples of the first type above. Triples of the second type as above were shown to be indeed heroic by Aboulker et al. (cf.~\cite{aboulker}, Theorem~4.1), their proof is based on the results from~\cite{harutyunyan2}. Finally, for the third and fourth types of triples we deal with oriented graphs. Let us restate the conjectures for these cases.
\begin{conjecture}\label{con:starforest}
For every orientation of a star forest $F$ and every hero $H$ the oriented graphs excluding $F$ and $H$ as induced subdigraphs have bounded dichromatic number.
\end{conjecture}

\begin{conjecture}\label{con:transitive}
For every orientation of a forest $F$ and every $k \in \mathbb{N}$ the oriented graphs excluding $F$ and $\vec{K}_k$ as induced subdigraphs have bounded dichromatic number.
\end{conjecture}

Conjecture~\ref{con:starforest} and Conjecture~\ref{con:transitive} can be regarded as oriented variants of the Gy\'{a}rf\'{a}s-Sumner conjecture, and in the main results of this paper we solve several special cases of these conjectures. Let us remark that Conjecture~\ref{con:transitive} can be reduced to the case in which $F$ is an oriented tree via the following observation. 

\begin{proposition}
Let $F_1$ and $F_2$ be oriented forests and let $F$ be their disjoint union. If Conjecture~\ref{con:transitive} holds for $F_1$ and $F_2$, then it also holds for $F$. 
\end{proposition}
\begin{proof}
Suppose that $\{\bivec{K}_2,F_i,\vec{K}_k\}$ is heroic for all $i \in \{1,2\}$ and $k \in \mathbb{N}$.

We prove by induction on $k \ge 2$ that digraphs in $\text{Forb}_\text{ind}(\bivec{K}_2, F, \vec{K}_k)$ have bounded dichromatic number. This is obvious for $k=2$, so let $k \ge 3$ and suppose there is $C \in \mathbb{N}$ such that every digraph in $\text{Forb}_\text{ind}(\bivec{K}_2, F, \vec{K}_{k-1})$ admits an acyclic $C$-coloring. Let $$C':=\max\{\vec{\chi}(\text{Forb}_\text{ind}(\bivec{K}_2,F_1,\vec{K}_k)), v(F_1)(1+2C)+\vec{\chi}(\text{Forb}_\text{ind}(\bivec{K}_2,F_2,\vec{K}_k))\}.$$ We claim that every digraph in $\text{Forb}_\text{ind}(\bivec{K}_2,F,\vec{K}_k)$ is acyclically $C'$-colorable. Suppose towards a contradiction that there exists $D \in \text{Forb}_\text{ind}(\bivec{K}_2,F,\vec{K}_k)$ such that $\vec{\chi}(D)>C'$. Since $C' \ge \vec{\chi}(\text{Forb}_\text{ind}(\bivec{K}_2,F_1,\vec{K}_k))$, we find that $D \notin \text{Forb}_\text{ind}(\bivec{K}_2, F_1, \vec{K}_k)$, and hence $D$ contains an induced copy of $F_1$. Let $X \subseteq V(D)$ be the vertex-set of this copy, let $Y$ be the set of vertices outside $X$ having at least one neighbor in $X$, and let us put $Z:=V(D)\setminus (X \cup Y)$. Then we have $\vec{\chi}(D[X]) \le |X|=v(F_1)$. By definition, $Y$ is contained in the union of the $2v(F_1)$ sets $N^+(x), x \in X$ and $N^-(x), x \in X$. It is now clear that since $D \in \text{Forb}_\text{ind}(\bivec{K}_2, F, \vec{K}_k)$, we must have $D[N^+(x)], D[N^-(x)] \in \text{Forb}_\text{ind}(\bivec{K}_2,F, \vec{K}_{k-1})$ for every $x \in X$, and it follows that $$\vec{\chi}(D[Y]) \le \sum_{x \in X}{(\vec{\chi}(D[N^+(x)])+\vec{\chi}(D[N^-(x)]))} \le 2v(F_1)C$$ by definition of $C$. Finally, since there are no arcs in $D$ connecting vertices of $X$ and $Z$, it follows that $F$ cannot contain an induced copy of $F_2$, as the vertex-set of this copy joined with $X$ would induce a copy of $F$ in $D$. Hence, we have $\vec{\chi}(D[Z]) \le \vec{\chi}(\text{Forb}_\text{ind}(\bivec{K}_2,F_2,\vec{K}_k))$. We conclude
$$\vec{\chi}(D) \le \vec{\chi}(D[X])+\vec{\chi}(D[Y])+\vec{\chi}(D[Z]) \le v(F_1)+2v(F_1)C+\vec{\chi}(\text{Forb}_\text{ind}(\bivec{K}_2,F_2,\vec{K}_k)) \le C',$$ contradicting our assumptions on $D$. This shows that these assumptions were wrong. It follows that $\{\bivec{K}_2, F, \vec{K}_k\}$ is heroic for all $k \ge 1$, as required.
\end{proof}

Aboulker et al.~\cite{aboulker} noted that in case that $H$ is a transitive tournament, Conjecture~\ref{con:starforest} is implied by a result of Chudnovsky, Scott and Seymour~\cite{chudnovsky2}. They further observed that Conjecture~\ref{con:starforest} is true in the case that $F$ has at most two vertices. Finally they focused on the case when $H=\vec{C}_3$ is the smallest non-trivial hero and $F$ has $3$ vertices. Then $F$ must be one of the following:
\begin{itemize}
\item $\overline{K}_3$, the forest consisting of three isolated vertices,
\item $\vec{P}_3$, the directed path on three vertices,
\item $\vec{K}_2+K_1$, the oriented star forest consisting of an arc plus an isolated vertex,
\item $S_2^+$, the $2$-out-star, or 
\item $S_2^-$, the $2$-in-star.
\end{itemize}
They proved that $\{\bivec{K}_2,F,\vec{C}_3\}$ is indeed heroic if $F$ is one of the first three star forests. Already in the case $F \in \{S_2^+,S_2^-\}$ however, the could not to prove heroicness and made the following explicit conjecture.

\begin{conjecture}[cf.~\cite{aboulker}, Conjecture 6.2]\label{con:S2+}
$$\vec{\chi}(\text{Forb}_{\text{ind}}(\bivec{K}_2,S_2^+,\vec{C}_3))=\vec{\chi}(\text{Forb}_{\text{ind}}(\bivec{K}_2,S_2^-,\vec{C}_3))=2.$$
\end{conjecture}
Note that by symmetry of reversing all arcs, it suffices to prove Conjecture~\ref{con:S2+} for the out-star $S_2^+$. The digraphs in $\text{Forb}_{\text{ind}}(\bivec{K}_2,S_2^+,\vec{C}_3)$ are exactly the directed triangle-free oriented graphs such that the out-neighborhood of every vertex induces a tournament. As the first main result of this paper, we prove Conjecture~\ref{con:S2+}.

\begin{theorem}\label{thm:dirtriangle}
$$\vec{\chi}(\text{Forb}_{\text{ind}}(\bivec{K}_2,S_2^+,\vec{C}_3))=\vec{\chi}(\text{Forb}_{\text{ind}}(\bivec{K}_2,S_2^-,\vec{C}_3))=2.$$
\end{theorem}

In fact, we deduce Theorem~\ref{thm:dirtriangle} as an immediate Corollary of the following stronger result involving the hero $W_3^+$.

\begin{theorem}\label{thm:transitive}
$\vec{\chi}(\text{Forb}_{\text{ind}}(\bivec{K}_2,S_2^+,W_3^+))=2$.
\end{theorem}

In order to prove Theorem~\ref{thm:transitive}, we need to establish several auxiliary results which deal with the structure of digraphs in the class $\text{Forb}_{\text{ind}}(\bivec{K}_2,S_2^+,W_3^+)$, which is surprisingly complicated (these are exactly the oriented graphs in which the out-neighborhood of every vertex induces a transitive tournament). 

As a next step we verify Conjecture~\ref{con:starforest} for more triples of the form $\{\bivec{K}_2,S_2^+, H\}$, where $H$ is some hero. We start with $W_3^-$.

\begin{theorem}\label{thm:W4minus}
$\vec{\chi}(\text{Forb}_{\text{ind}}(\bivec{K}_2,S_2^+,W_3^-)) \le 4$.
\end{theorem}

Our last result concerning Conjecture~\ref{con:starforest} generalizes Theorem~\ref{thm:W4minus} qualitatively and proves that for every $k \in \mathbb{N}$, the triple $\{\bivec{K}_2,S_2^+,H_k\}$ is heroic, where $H_k$ is the hero on $k$ vertices obtained from the disjoint union of $W_3^+$ and $\vec{K}_{k-4}$ by adding all possible arcs from $W_3^+$ towards $\vec{K}_{k-4}$. More generally, we show the following.
\begin{theorem}\label{thm:addsink}
Let $H$ be a hero and let $H^-$ be the hero obtained from $H$ by adding a dominating sink. 
If $\{\bivec{K}_2,S_2^+,H\}$ is heroic, then so is $\{\bivec{K}_2,S_2^+, H^-\}$.
\end{theorem}

Our last new result in this paper concerns Conjecture~\ref{con:transitive}. As mentioned above, Conjecture~\ref{con:transitive} holds true whenever $F$ is an oriented star forest, and therefore particularly for forests on at most $3$ vertices. The first open cases therefore appear when $F$ is an orientation of the $P_4$. Aboulker et al.~considered the directed path $\vec{P}_4$ and showed in one of their main results that the set $\{\bivec{K}_2,\vec{P}_4,\vec{K}_3\}$ is heroic. There are three other oriented paths on four vertices. Two of them, which are called $P^+(2,1)$ and $P^-(2,1)$ in~\cite{aboulker}, consist of two oppositely oriented dipaths of length two and one, respectively. Chudnovsky, Scott and Seymour proved in~\cite{chudnovsky2} that for every $k \in \mathbb{N}$, digraphs in the set $\text{Forb}_{\text{ind}}(\bivec{K}_2,P,\vec{K}_k)$ have underlying graphs with bounded chromatic number (and thus bounded dichromatic number) for $P \in \{P^+(2,1),P^-(2,1)\}$. Hence, Conjecture~\ref{con:transitive} holds for these two orientations of $P_4$. The same result however is wrong for the remaining orientation of $P_4$, denoted $P^+(1,1,1)$ in~\cite{aboulker}, as it consists of $3$ alternatingly oriented arcs. Here we complement the result of Aboulker et al.~\cite{aboulker} concerning the directed path $\vec{P}_4$ and $k=3$ by showing that also the set $\{\bivec{K}_2,P^+(1,1,1),\vec{K}_3\}$ is heroic.

\begin{theorem}\label{thm:N}
$\vec{\chi}(\text{Forb}_\text{ind}(\bivec{K}_2,P^+(1,1,1),\vec{K}_3))=2$.
\end{theorem}
We remark that the class $\text{Forb}_\text{ind}(\bivec{K}_2,P^+(1,1,1),\vec{K}_3)$ is quite rich, as it (among others) contains all oriented \emph{line digraphs}.

\paragraph{Structure of the paper.}
In Section~\ref{sec:outransitive} we investigate the structure of digraphs in the class $\text{Forb}_{\text{ind}}(\bivec{K}_2,S_2^+,W_3^+)$ and use these insights to prove Theorem~\ref{thm:transitive}. In Section~\ref{sec:W3-} we give the proof of Theorem~\ref{thm:W4minus}. In Section~\ref{sec:addsink} we prove Theorem~\ref{thm:addsink}. Finally, in Section~\ref{sec:4path} we prove Theorem~\ref{thm:N} and we conclude with final comments in Section~\ref{sec:inducedconclusion}.

\section{$\{S_2^+, W_3^+\}$-Free Oriented Graphs}\label{sec:outransitive}
In this section, we will prove Theorem~\ref{thm:transitive} and thereby show that $\{\bivec{K}_2,S_2^+,W_3^+\}$ is a heroic set. Note that $\text{Forb}_{\text{ind}}(\bivec{K}_2,S_2^+,W_3^+)$ is the class of oriented graphs $D$ with the property that the out-neighbourhood of every vertex in $D$ induces a transitive tournament. Given $D \in \text{Forb}_{\text{ind}}(\bivec{K}_2,S_2^+,W_3^+)$, we define $F = F(D)$ to be the spanning subdigraph of $D$ consisting of the arcs $(x,y) \in A(D)$ such that $y$ is the source in the transitive tournament induced by the out-neighbourhood of $x$ in $D$. Observe that for every $x \in V(D)$, if $d^+_D(x) \geq 1$ then $d^+_F(x) = 1$, and otherwise $d^+_F(x) = 0$. From the definition of $F(D)$ we immediately obtain  the following property:
	\begin{claim}\label{claim:F}
		Let $D \in \text{Forb}_{\text{ind}}(\bivec{K}_2,S_2^+,W_3^+)$ and $(x,y) \in A(F(D))$. Then we have $$N^+_D(x) \subseteq N^+_D(y) \cup \{y\}.$$
	\end{claim}

	\noindent
	The next claim follows immediately from Claim \ref{claim:F} via induction.

	\begin{claim}\label{claim:F_dipath}
		Let $D \in \text{Forb}_{\text{ind}}(\bivec{K}_2,S_2^+,W_3^+)$ and let $(x_1,\dots,x_k)$ be a dipath in $F(D)$. Then $$N^+_D(x_1)\setminus \{x_2,\ldots,x_k\} \subseteq N^+_D(x_k).$$
	\end{claim}

	\noindent
	From Claim \ref{claim:F_dipath} we can derive that the vertex-sets of directed cycles in $F(D)$ form so-called \emph{out-modules} in $D$. 
	\begin{definition}
	Let $D$ be a digraph, and $\emptyset \neq M \subseteq V(D)$. We say that $M$ is an \emph{out-module} in $D$ if it holds that $(x,z) \in A(D) \Rightarrow (y,z) \in V(D)\setminus M$ for every $x, y \in M$ and $z \in V(D)\setminus M$. Equivalently, $N_D^+(x)\setminus M=N_D^+(y)\setminus M$ for all $x, y \in M$.
	\end{definition}

	\begin{claim}\label{claim:cycle_out-module}
		Let $D \in \text{Forb}_{\text{ind}}(\bivec{K}_2,S_2^+,W_3^+)$, and let $C$ be a directed cycle in $F(D)$. Then $V(C)$ is an out-module in $D$.
	\end{claim}
	\begin{proof}
	Let $x_1,x_2,\ldots,x_k,x_1$ be the vertex-trace of $C$.
		Let $y \in V(D) \setminus \{x_1,\dots,x_k\}$ and $1 \leq i \leq k$ be arbitrary such that $(x_i,y) \in A(D)$. Let $j \in [k] \setminus \{i\}$.  
		By Claim \ref{claim:F_dipath}, applied to the directed  subpath of $C$ starting in $x_i$ and ending in $x_j$, we know that $N^+_D(x_i)\setminus \{x_1,\ldots,x_k\} \subseteq N^+_D(x_j)\setminus \{x_1,\ldots,x_k\}$. Hence $(x_j,y) \in A(D)$. This shows that $\{x_1,\dots,x_k\}$ is indeed an out-module.  
	\end{proof}


	
	For a non-empty vertex-set $U$ in a digraph $D$, we denote by $D/U$ the digraph obtained by \emph{identifying} $U$, that is, the digraph with vertex set $(V(D) \setminus U) \cup \{x_U\}$ where $x_U \notin V(D)$ is some newly added vertex representing $U$, and the following arcs: the arcs of $D$ inside $V(D) \setminus U$, the arc $(x_U,v)$ for every $v \in N^+_D(U)$, and the arc $(v,x_U)$ for every $v \in N^-_D(U)$. 
	
In the following we prepare the proof of Theorem~\ref{thm:transitive} with a set of useful Lemmas.
We start with two lemmas yielding modifications of digraphs which preserve the containment in the class $\text{Forb}_\text{ind}(\bivec{K}_2,S_2^+,W_3^+)$.
	\begin{lemma}\label{lem:out-module_contraction}
		For every $D \in \text{Forb}_\text{ind}(\bivec{K}_2,S_2^+,W_3^+)$ and for every out-module $U \subseteq V(D)$ it holds that $D/U \in \text{Forb}_\text{ind}(\bivec{K}_2,S_2^+,W_3^+)$.
	\end{lemma}
	\begin{proof}
		We need to show that $D/U$ is induced $\{\bivec{K}_2,S_2^+,W_3^+\}$-free. We argue by contradiction.
		Suppose first that $D/U$ contains a $\bivec{K}_2$, namely a pair of vertices $x,y$ with $(x,y),(y,x) \in A(D/U)$. If $x,y \neq x_U$ then $x,y$ also span a copy of $\bivec{K}_2$ in $D$, a contradiction. Suppose then that $x = x_U$ or $y = x_U$; say $x = x_U$. By the definition of $D/U$, there are (not necessarily distinct) $u_1,u_2 \in U$ such that $(u_1,y),(y,u_2) \in A(D)$. Since $U$ is an out-module, $(u_2,y) \in A(D)$. Hence, $u_2,y$ span a copy of $\bivec{K}_2$ in $D$, a contradiction.
		
		Suppose next that $D/U$ contains an induced copy of $S_2^+$, namely, distinct vertices $x,y,z \in V(D/U)$ with $(x,y),(x,z) \in A(D/U)$ and with no arc in $D/U$ between $y$ and $z$. If $x,y,z \neq x_U$ then $x,y,z$ also span an induced $S_2^+$ in $D$, a contradiction. Suppose now that $y = x_U$, and let $u \in U$ be such that $(x,u) \in A(D)$. We have $(u,z),(z,u) \notin A(D)$ because $(x_U,z),(z,x_U) \notin A(D/U)$. Hence, $x,u,z$ span an induced $S_2^+$ in $D$, a contradiction. The case $z = x_U$ is analogous. Suppose now that $x = x_U$. Since $(x_U,y),(x_U,z) \in A(D/U)$ and $U$ is an out-module, we must have $(u,y),(u,z) \in A(D)$ for every $u \in U$, implying that $u,y,z$ span an induced $S_2^+$ in $D$ for every such $u$, again yielding the desired contradiction. 
		
		Suppose now that $D/U$ contains a copy of $W_3^+$ with vertices $x,y,z,w$ and arcs \linebreak $(x,y),(x,z),(x,w),(y,z),(z,w),(w,y)$. Again, if $x,y,z,w \neq x_U$ then $D$ also has a copy of $W_3^+$, a contradiction. Suppose now that $x = x_U$, and fix any $u \in U$. Since $U$ is an out-module, we have 
		$(u,y),(u,z),(u,w) \in A(D)$, implying that $u,y,z,w$ span a copy of $W_3^+$ in $D$, a contradiction. Suppose finally that one of $y,z,w$ equals $x_U$, say $y = x_U$ (without loss of generality). Since $(x,x_U),(x_U,z),(w,x_U) \in A(D/U)$, there are $u_1,u_2,u_3 \in U$ (not necessarily distinct) such that $(x,u_1),(u_2,z),(w,u_3) \in A(D)$. Since $U$ is an out-module, we have $(u_1,z),(u_3,z) \in A(D)$. Since $(x,w),(x,u_1) \in A(D)$ and $D$ is induced $S_2^+$-free, we must have either $(w,u_1) \in A(D)$ or $(u_1,w) \in A(D)$. If $(u_1,w) \in A(D)$ then also $(u_3,w) \in A(D)$ because $U$ is an out-module, but this is impossible as then $u_3,w$ would induce a digon in $D$. Finally, if $(w,u_1) \in A(D)$ then $x,u_1,z,w$ span a copy of $W_3$ in $D$, again yielding a contradiction. This concludes the proof of the lemma.
		\end{proof}
		
\begin{lemma}\label{lem:arcaddition}
Let $D \in \text{Forb}_\text{ind}(\bivec{K}_2,S_2^+,W_3^+)$, and let $(x,y) \in A(F(D))$. Let $z \in N_D^+(y)$ such that $(x,z), (z,x) \notin A(D)$. Then the digraph $D+(x,z)$ obtained from $D$ by adding the arc $(x,z)$ is contained in $\text{Forb}_\text{ind}(\bivec{K}_2,S_2^+,W_3^+)$.
\end{lemma}
\begin{proof}
We need to show that $D+(x,z)$ is induced $\{\bivec{K}_2,S_2^+,W_3^+\}$-free. Again, we argue by contradiction. Clearly $D+(x,z)$ does not contain a $\bivec{K}_2$, since $(z,x) \notin A(D)$ by assumption. Suppose next that $D+(x,z)$ contains an induced copy of $S_2^+$, i.e. distinct vertices $a,b,c$ such that $(a,b),(a,c) \in A(D+(x,z))$, and $(b,c),(c,b) \notin A(D+(x,z))$. If $(x,z) \notin \{(a,b), (a,c)\}$, then $a,b,c$ induce a copy of $S_2^+$ also in $D$, a contradiction. We may therefore assume w.l.o.g. that $(x,z)=(a,b)$. Then we have $c \neq y$, since $(c,b) \notin A(D)$, but $(y,b)=(y,z) \in A(D)$ by assumption. Since $(x,c)=(a,c) \in A(D)$, we have $c \in N_D^+(x)$. But $(x,y) \in A(F(D))$, and hence Claim~\ref{claim:F} implies that $c \in N_D^+(x) \subseteq \{y\} \cup N_D^+(y)$. It follows that $(y,c) \in A(D)$. We further have $(y,b)=(y,z) \in A(D)$ and $(b,c), (c,b) \notin A(D)$. Hence, $y,b,c$ induce an $S_2^+$ in $D$, a contradiction.

Moving on, suppose that $D+(x,z)$ contains an induced copy of $W_3^+$, i.e., distinct vertices $a,b,c,d$ such that $(a,b),(a,c),(a,d),(b,c),(c,d),(d,b) \in A(D) \cup \{(x,z)\}$. 

Suppose first that $(x,z) \notin \{(a,b),(a,c),(a,d)\}$. Then $(a,b),(a,c),(a,d) \in A(D)$ and since $D$ does not contain an induced copy of $S_2^+$, the vertices $b,c,d$ are pairwise adjacent. Since $x$ and $z$ are non-adjacent in $D$, it follows that $\{x,z\} \not\subseteq\{b,c,d\}$. Hence we have $(b,c),(c,d),(d,b) \in A(D)$, yielding that $a,b,c,d$ induce a copy of $W_3^+$ in $D$, a contradiction.
Hence we may suppose that $(x,z) \in \{(a,b),(a,c),(a,d)\}$. By symmetry we may assume that $(x,z)=(a,b)$  w.l.o.g. Again using Claim~\ref{claim:F} we then have $c,d \in N_D^+(x) \subseteq N_D^+(y) \cup \{y\}$.

Let us first consider the case that $c,d \neq y$. Then $(y,c),(y,d) \in A(D)$, and since $(y,b)=(y,z) \in A(D)$ by assumption, it follows that the vertices $y,b,c,d$ induce a copy of $W_3^+$ in $D$, a contradiction. For the next case suppose that $y \in \{c,d\}$. The first option, namely that $y=c$, is impossible, since then we would have $(y,z) \in A(D)$ (by assumption) and $(z,y)=(b,c) \in A(D)$, a contradiction since $D$ is $\bivec{K}_2$-free. Therefore, we must have $y=d$. Then $c \in N_D^+(y)$ as well as $(c,y)=(c,d) \in A(D)$. It follows that $y,c$ induce a $\bivec{K}_2$ in $D$, so again, we conclude with a contradiction.

Having reached a contradiction in all cases, it follows that our initial assumption was wrong, indeed, $D+(x,z) \in \text{Forb}_\text{ind}(\bivec{K}_2,S_2^+,W_3^+)$. This concludes the proof. 
\end{proof}

The next lemma shows the existence of out-modules with special properties.
\begin{lemma}\label{lem:out-moduleinin}
Let $D \in \text{Forb}_\text{ind}(\bivec{K}_2,S_2^+,W_3^+)$, and let $v \in V(D)$. If $N_D^-(v) \neq \emptyset$, then there exists an out-module $M$ in $D$ such that $M \subseteq N_D^-(v)$ and $N_D^+(M) \subseteq N_D^+(v) \cup \{v\}$.
\end{lemma}
\begin{proof}
We prove by induction on $n \ge 1$ the statement of the lemma for all digraphs $D \in \text{Forb}_\text{ind}(\bivec{K}_2,S_2^+,W_3^+)$ and vertices $v \in V(D)$ such that $d_D^-(v) \le n$.

If $n=1$, then $N_D^-(v)=\{w\}$ for a vertex $w \in V(D)$. Then $M:=\{w\}$ is an out-module of $D$. 
Hence, it suffices to verify that $N_D^+(w) \subseteq N_D^+(v) \cup \{v\}$. 
Suppose towards a contradiction that $(w,v') \in A(D)$ for $v' \in N_D^+(w)\setminus (N_D^+(v) \cup \{v\})$. Since $N_D^-(v)=\{w\}$, we have $v' \notin N_D^-(v)$, and hence $v,v'$ are non-adjacent in $D$, while $(w,v),(w,v') \in A(D)$. Hence, $w,v,v'$ induce an $S_2^+$ in $D$, a contradiction.

Now let $D \in \text{Forb}_\text{ind}(\bivec{K}_2,S_2^+,W_3^+)$ and $v \in V(D)$ such that $d_D^-(v)=n \ge 2$, and assume that the claim holds for all pairs of digraphs in $\text{Forb}_\text{ind}(\bivec{K}_2,S_2^+,W_3^+)$ and vertices whose in-degree is less than $n$.

First let us assume that there exists a vertex $w \in N_D^-(v)$ such that $(w,v) \in A(F(D))$. Then $M:=\{w\}$ is an out-module of $D$, and by Claim~\ref{claim:F} we have
$N_D^+(w) \subseteq N_D^+(v) \cup \{v\}$. This proves the assertion in this case.

Hence, for the rest of this proof we may suppose that $(w,v) \notin A(F(D))$ for every $w \in N_D^-(v)$. For any $w \in N_D^-(v)$ we clearly have $d_D^+(w) \ge 1$  and hence it follows that $d_{F(D)}^+(w)=1$. Furthermore, for every arc $(w,w') \in A(F(D))$ such that $w \in N_D^-(v)$ by Claim~\ref{claim:F} we must have $v \in N_D^+(w) \subseteq N_D^+(w') \cup \{w'\}$. Since $(w,v) \notin A(F(D))$, we have $v \neq w'$ and hence $(w',v) \in A(D)$. This shows that the out-neighbor in $F(D)$ of any vertex in $N_D^-(v)$ is again contained in $N_D^-(v)$. 
It follows that $F(D)$ restricted to $N_D^-(v)$ has minimum out-degree $1$ and therefore contains a directed cycle $C$ such that $V(C) \subseteq N_D^-(v)$. 

By Claim~\ref{claim:cycle_out-module}, $N:=V(C)$ is an out-module in $D$. Consider the digraph $D':=D/N$, which by Lemma~\ref{lem:out-module_contraction} is a member of $\text{Forb}_\text{ind}(\bivec{K}_2,S_2^+,W_3^+)$. Then by definition of $D/N$ and since $N \subseteq N_D^-(v)$, we have $v \in V(D')$ and $N_{D'}^-(v)=(N_D^-(v)\setminus N) \cup \{x_N\} \neq \emptyset$. Since $|N|=|V(C)| \ge 3$, this implies that $d_{D'}^-(v)=1+d_D^-(v)-|N| \le d_D^-(v)-2<n$. We may therefore apply the induction hypothesis to the digraph $D' \in \text{Forb}_\text{ind}(\bivec{K}_2,S_2^+,W_3^+)$ and the vertex $v \in V(D')$. We thus find an out-module $M'$ in $D'$ with the properties $M' \subseteq N_{D'}^-(v)=(N_D^-(v)\setminus N)\cup \{x_N\}$ and $N_{D'}^+(M') \subseteq N_{D'}^+(v) \cup \{v\}=N_D^+(v) \cup \{v\}$. 
Let us define the set $M \subseteq N_D^-(v)$ as $M:=M'$, if $x_N \notin M'$, and $M:=(M' \setminus \{x_N\}) \cup N$ if $x_N \in M'$. We claim that $M$ satisfies the assertions of the Lemma with respect to $D$ and $v$. In the following, we verify both parts of the inductive claim separately.
\paragraph{Claim 1.} $N_D^+(M) \subseteq N_D^+(v) \cup \{v\}$.
\begin{proof}
In the proof we will use the fact that
$$N_{D'}^+(M') \subseteq N_{D'}^+(v) \cup \{v\}=N_D^+(v) \cup \{v\},$$
which holds by the induction hypothesis.  

Let $x \in N_D^+(M)$ be given arbitrarily. Let $m \in M$ such that $(m,x) \in A(D)$. Our goal is to show that $x \in N_D^+(v) \cup \{v\}$.

Let us first consider the case that $x_N \notin M'$ and hence $M=M'$. By definition of $D'=D/N$ we either have $x \notin N$ and $(m,x) \in A(D')$, or $x \in N$ and $(m,x_N) \in A(D')$. Then since $m \in M=M'$ and $x, x_N \notin M=M'$, we obtain that either $x \notin N$ and 
$x \in N^+_{D'}(M')$,
or $x \in N$ and 
$x_N \in N^+_{D'}(M')$.
As 
$N_{D'}^+(M') \subseteq N_D^+(v) \cup \{v\}$, 
in the first case we have 
$x \in N_D^+(v) \cup \{v\}$,
as desired. 
The second case does not occur, since it yields $x_N \in N_D^+(v) \cup \{v\}$, which is impossible as $x_N$ is not a vertex of $D$. Hence, we have shown the claim that $x \in N_D^+(v) \cup \{v\}$.

For the second case, suppose that $x_N \in M'$ and hence $M=(M'\setminus \{x_N\}) \cup N$. Note that $x \notin N$, since $x \notin M \supseteq N$. Hence, the existence of the arc $(m,x) \in A(D)$ yields that either $m \in N$ and $(x_N,x) \in A(D')$, or $m \notin N$ and $(m,x) \in A(D')$. In both cases, this implies that $x \in N_{D'}^+(M') \subseteq N_D^+(v) \cup \{v\}$, 
proving the assertion.
\end{proof}

\paragraph{Claim 2.} $M$ is an out-module in $D$.
\begin{proof}
Let $x \neq y \in M$ and $z \in V(D) \setminus M$ arbitrary, and assume that $(x,z) \in A(D)$. We need to show that also $(y,z) \in A(D)$. Note that by Claim 1 we have $z \notin N_D^-(v)$, as otherwise $z \in N_D^+(M) \cap N_D^-(v)=\emptyset$. In particular, $z \notin N$.

Observe that $z \in N^+_{D'}(M')$. Indeed, if $x \notin N$ then $x \in M'$ and $(x,z) \in A(D')$, and if $x \in N$ then $x_N \in M'$ and $(x_N,z) \in A(D')$; in any case, $z \in N^+_{D'}(M')$. 

Since $y \in M$, we have either $y \in M'$ or $y \in N$ and $x_N \in M'$. Suppose first that $y \in M'$. Then $(y,z) \in A(D')$ because $z \in N^+_{D'}(M')$ and $M'$ is an out-module. Hence, in this case $(y,z) \in A(D)$, as required. Now suppose that $y \in N$ and $x_N \in M'$. Since $z \in N^+_{D'}(M')$ and $M'$ is an out-module, we have $(x_N,z) \in A(D')$. This means that there is $w \in N$ such that $(w,z) \in A(D)$. Now, as $N$ is itself is an out-module in $D$ and $y \in N, z \notin N$, we have $(y,z) \in A(D)$, as required. 
\end{proof}

By Claim 1 and 2 the out-module $M$ certifies that the pair $(D,v)$ satisfies the inductive claim. This concludes the proof of the Lemma by induction.
\end{proof}

\begin{lemma}\label{lem:twosteps}
Let $D \in \text{Forb}_\text{ind}(\bivec{K}_2,S_2^+,W_3^+)$, let $M \subseteq V(D)$ be an out-module in $D$ and let $v \in V(D)\setminus M$. Let $T$ be the set of vertices defined by
$$T:=\{t \in M|\exists u \in V(D)\setminus M: (v,u), (u,t) \in A(D)\}.$$
Then $D[T]$ is a (possibly empty) transitive tournament.
\end{lemma}
\begin{proof}
The assertion will follow directly from the following two claims.
\paragraph{Claim 1.} If $t_1 \neq t_2 \in T$, then $t_1$ and $t_2$ are adjacent in $D$.
\begin{proof}
By definition of $T$ there exist vertices $u_1, u_2 \in V(D)\setminus M$ (not necessarily distinct) such that $(v,u_i), (u_i,t_i) \in A(T)$ for $i=1,2$. If $u_1=u_2$, then $t_1,t_2$ must be adjacent, for otherwise the vertices $u_1, t_1, t_2$ would induce an $S_2^+$ in $D$, a contradiction. Suppose now that $u_1 \neq u_2$. Since $u_1, u_2 \in N^+(v)$, they must be adjacent, w.l.o.g. let $(u_1,u_2) \in A(D)$. Then $t_1 \in M$ and $u_2 \in V(D)\setminus M$ are distinct out-neighbors of $u_1$, and hence they must be adjacent in $D$. If $(t_1,u_2) \in A(D)$, then $M$ being an out-module implies that also $(t_2,u_2) \in A(D)$, yielding a $\bivec{K}_2$ in $D$ induced by $t_2$ and $u_2$, a contradiction. Therefore we have $(u_2,t_1) \in A(D)$. Then $t_1$ and $t_2$ are distinct out-neighbors of $u_2$ in $D$, which implies that they must be adjacent. This concludes the proof.
\end{proof}
\paragraph{Claim 2.} $D[T]$ contains no directed triangle.
\begin{proof}
Suppose towards a contradiction that there are three distinct $t_1, t_2, t_3 \in T$ inducing a directed triangle in $D$. Let $u_1, u_2, u_3 \in V(D)\setminus M$ be (not necessarily distinct) such that $(v,u_i), (u_i,t_i) \in A(D), i=1,2,3$. We distinguish three different cases depending on the size of the set $\{u_1,u_2,u_3\}$. 

For the first case, suppose that $u_1=u_2=u_3$. Then $t_1, t_2, t_3$ are three distinct out-neighbors of $u_1$ spanning a directed triangle. Hence, $u_1, t_1, t_2, t_3$ induce a $W_3^+$ in $D$, a contradiction to our assumption that $D \in \text{Forb}_\text{ind}(\bivec{K}_2,S_2^+,W_3^+)$.

For the second case, suppose that $\{u_1,u_2,u_3\}$ contains exactly two distinct vertices, w.l.o.g. $u_1 \neq u_2=u_3$. Since $u_1$ and $u_2$ are two distinct out-neighbors of $v$ in $D$, they must be adjacent. Suppose first that $(u_1,u_2) \in A(D)$. Then $t_1$ and $u_2$ are two distinct out-neighbors of $u_1$ in $D$, and hence they must be adjacent. If $(t_1,u_2) \in A(D)$, then by the module-property of $M$, also $(t_2,u_2) \in A(D)$, and hence $t_2, u_2$ induce a $\bivec{K}_2$ in $D$, a contradiction. If $(u_2,t_1) \in A(D)$, then $t_1, t_2, t_3 \in N_D^+(u_2)$ and hence $u_2,t_1,t_2,t_3$ induce a $W_3^+$ in $D$, a contradiction. 
Next suppose that $(u_2,u_1) \in A(D)$. Then $u_1, t_2, t_3$ are three distinct out-neighbors of $u_2$ in $D$, and hence $u_1$ must be adjacent to both $t_2$ and $t_3$. If $(t_2,u_1) \in A(D)$ or $(t_3,u_1) \in A(D)$, then $(t_1,u_1) \in A(D)$ since $M$ is an out-module, and hence $u_1, t_1$ induce a $\bivec{K}_2$ in $D$, a contradiction. Otherwise, we have $(u_1,t_2),(u_1,t_3) \in A(D)$ and hence $u_1, t_1, t_2, t_3$ induce a $W_3^+$ in $D$, again yielding the desired contradiction. 

For the third case, suppose that $u_1, u_2, u_3$ are pairwise distinct. Since $u_1, u_2, u_3$ are three distinct vertices in the transitive tournament $D[N^+(v)]$, they form a transitive triangle, and we may assume w.l.o.g. that $(u_1, u_2), (u_1, u_3), (u_2, u_3) \in A(D)$. Then $t_1$ and $u_3$ are distinct out-neighbors of $u_1$ in $D$, while $t_2$ and $u_3$ are distinct out-neighbors of $u_2$ in $D$. Hence, $u_3$ must be adjacent to both $t_1$ and $t_2$. If $(t_i,u_3) \in A(D)$ for some $i=1,2$, then we also have $(t_3,u_3) \in A(D)$ since $M$ is an out-module, and hence $u_3, t_3$ induce a $\bivec{K}_2$ in $D$, contradiction. Finally, if $(u_3,t_1),(u_3,t_2) \in A(D)$, then $u_3, t_1, t_2, t_3$ induce a $W_3^+$ in $D$, yielding again a contradiction to the containment of $D$ in $\text{Forb}_\text{ind}(\bivec{K}_2,S_2^+,W_3^+)$.

Since we arrived at a contradiction in each case, we conclude that the initial assumption concerning the existence of $t_1, t_2, t_3$ was wrong. This concludes the proof.
\end{proof}
\phantom\qedhere
\end{proof}

We are now sufficiently prepared to give the proof of Theorem~\ref{thm:transitive}. In fact, we will prove the following slightly stronger version of the result, which allows to enforce a monochromatic coloring on the closed out-neighborhood of an arbitrarily chosen vertex.

\begin{theorem}\label{thm:transitiveprecoloring}
Let $D \in \text{Forb}_\text{ind}(\bivec{K}_2,S_2^+,W_3^+)$, and $v \in V(D)$. Then there exists an acyclic coloring $c:V(D) \rightarrow \{1,2\}$ of $D$ such that $c(u)=c(v)$ for every $u \in N_D^+(v)$.
\end{theorem}
\begin{proof}
Suppose towards a contradiction that the claim is wrong, and let $D$ be a counterexample to the claim minimizing $v(D)$. Let $v \in V(D)$ be a vertex such that $D$ does not admit an acyclic $2$-coloring $c$ with the property that $c(u)=c(v)$ for every $u \in N_D^+(v)$. 
\paragraph{Claim 1.} $N_D^-(v) \neq \emptyset$.
\begin{proof}
Suppose towards a contradiction that $N_D^-(v)=\emptyset$. If also $N_D^+(v)=\emptyset$, then $v$ is an isolated vertex of $D$. Then any acyclic $2$-coloring of $D-v$ could be extended to an acyclic $2$-coloring of $D$ by coloring $v$ with color $1$, and the statement that $v$ has the same color as its out-neighbors would hold vacuously. Since this is impossible, we must have $\vec{\chi}(D-v) \ge 3$, which however contradicts the minimality of $D$ as a counterexample. This shows that $N_D^+(v) \neq \emptyset$. Let $u \in N_D^+(v)$ be the unique out-neighbor of $v$ in $F(D)$. Then $N_D^+(v) \subseteq N_D^+(u) \cup \{u\}$ by Claim \ref{claim:F}. The minimality of $D$ as a counterexample now implies that the digraph $D-v \in \text{Forb}_\text{ind}(\bivec{K}_2,S_2^+,W_3^+)$ admits an acyclic $2$-coloring $c^-:V(D) \rightarrow \{1,2\}$ satisfying $c^-(x)=c^-(u)$ for every $x \in N_{D-v}^+(u)=N_D^+(u)$. Let $c:V(D) \rightarrow \{1,2\}$ be defined as $c(x):=c^-(x)$ for every $x \in V(D)\setminus \{v\}$ and $c(v):=c^-(u)$. Since $c$ restricted to $V(D)\setminus \{v\}$ is an acyclic coloring, and no directed cycle in $D$ contains $v$ (recall $N_D^-(v)=\emptyset$), it follows that $c$ is an acyclic coloring of $D$. Moreover, for every $x \in N_D^+(v) \subseteq N_D^+(u) \cup \{u\}$ we have $c(x) = c^-(x) = c^-(u)=c(v)$. This is a contradiction to our initial assumption that $D$ does not admit an acyclic $2$-coloring with this property. This shows that our assumption $N_D^-(v)=\emptyset$ was wrong, concluding the proof.
\end{proof}
By Lemma~\ref{lem:out-moduleinin} applied to the vertex $v$ of $D$, there exists an out-module $M$ in $D$ such that $M \subseteq N_D^-(v)$ and $N_D^+(M) \subseteq N_D^+(v) \cup \{v\}$. Let $T \subseteq M$ be the set of vertices $t \in M$ for which there exists $u \in N_D^+(v)$ such that $(u,t) \in A(D)$. Since $N_D^+(v) \cap M=\emptyset$, the definition of $T$ here coincides with the one in Lemma~\ref{lem:twosteps}. Now, Lemma~\ref{lem:twosteps} implies that $D[T]$ is a (possibly empty) transitive tournament. 
\paragraph{Claim 2.} The digraph $D[M]$ admits an acyclic $2$-coloring $c_M:M \rightarrow \{1,2\}$ satisfying $c_M(t)=2$ for all $t \in T$.
\begin{proof}
Since $v(D[M]) \le v(D-v) <v(D)$, the minimality of the counterexample $D$ implies that $D[M] \in \text{Forb}_\text{ind}(\bivec{K}_2,S_2^+,W_3^+)$ satisfies the assertion of Theorem~\ref{thm:transitiveprecoloring}. If $T=\emptyset$, Claim~2 is satisfied by an arbitrary choice of an acyclic $2$-coloring for $D[M]$. If $T \neq \emptyset$, let $t_0 \in T$ be the source of the transitive tournament $D[T]$. Applying the assertion of the theorem to $D[M]$ and the vertex $t_0$, we find that there exists an acyclic $2$-coloring of $D[M]$ in which $t_0$ has the same color as all its out-neighbors. W.l.o.g. we may choose this color to be $2$, and since $\{t_0\} \cup N_{D[M]}^+(t_0) \supseteq T$, the claim follows.
\end{proof}
\paragraph{Claim 3.} $D[M]$ contains a directed cycle. 
\begin{proof}
Suppose towards a contradiction that $D[M]$ is acyclic. Let $D':=D-M$. Clearly, $D' \in \text{Forb}_\text{ind}(\bivec{K}_2,S_2^+,W_3^+)$. Since $v(D') \le v(D)-1$ and by the minimality of $D$ as a counterexample, we know that $D'$ admits an acyclic $2$-coloring $c':V(D)\setminus M \rightarrow \{1,2\}$ in which $c'(v)=c'(u)=1$ for every $u \in N_D^+(v)$. Let $c:V(D) \rightarrow \{1,2\}$ be defined by $c(x):=c'(x)$ for all $x \in V(D)\setminus M$ and $c(x):=2$ for all $x \in M$. We claim that $c$ is an acyclic coloring of $D$. Suppose towards a contradiction that $C$ is a directed cycle in $D$ which is monochromatic in the coloring $c$. We must have $V(C) \cap M \neq \emptyset$, for otherwise $C$ would form a monochromatic directed cycle in the coloring $c'$ of $D'$. Since $D[M]$ is acyclic, we must also have $V(C)\setminus M\neq \emptyset$. It follows that there exists an arc $(x,y) \in A(C)$ such that $x \in M$ and $y \in V(D)\setminus M$. Then $y \in N_D^+(M) \subseteq N_D^+(v) \cup \{v\}$, and therefore $c(y)=c'(y)=1$, while $c(x)=2$ by definition. This contradicts the fact that $C$ is monochromatic, and hence we have shown that indeed $c$ is an acyclic coloring of $D$. Moreover, $c(v)=c(u)=1$ for every $u \in N_D^+(v)$. This contradicts our initial assumptions on $D$ that such a coloring does not exist. Hence, $D[M]$ cannot be acyclic, proving the claim.
\end{proof}
Claim 3 in particular implies that $|M| \ge 3$ and $M \setminus T \neq \emptyset$.

Let us further note that since $M$ forms an out-module in $D$, $M \setminus T \neq \emptyset$ is an out-module in the digraph $D-T \in \text{Forb}_\text{ind}(\bivec{K}_2,S_2^+,W_3^+)$, and hence by Lemma~\ref{lem:out-module_contraction} we also have $D_0:=(D-T)/(M\setminus T) \in \text{Forb}_\text{ind}(\bivec{K}_2,S_2^+,W_3^+)$. Also note that since $T \subseteq M \subseteq N_D^-(v)$, we still have $N_D^+(v) \cup \{v\} \subseteq \{x_{M\setminus T}\} \cup (V(D)\setminus M)=V(D_0)$, where we denote by $x_{M\setminus T}$ the vertex in $D_0$ obtained by identifying $M\setminus T$.
\paragraph{Claim 4.} We have $N_{D_0}^+(v)=N_D^+(v)$, $(x_{M\setminus T},v) \in A(F(D_0))$, and for every $u \in N_{D_0}^+(v)$, we have $(u,x_{M\setminus T}) \notin A(D_0)$.
\begin{proof}
The very first claim follows directly from the definition of $D_0$.

We have $M\setminus T \subseteq N_D^-(v)$ and $N_D^+(M) \subseteq \{v\} \cup N_D^+(v)$. This directly implies that $(x_{M\setminus T},v) \in A(D_0)$ and that $N_{D_0}^+(x_{M\setminus T}) \subseteq N_D^+(M) \subseteq N^+(v) \cup \{v\} =N_{D_0}^+(v) \cup \{v\}$. Hence, $v \in N_{D_0}^+(x_{M\setminus T})$ has an out-arc to every other out-neighbor of $x_{M\setminus T}$ in $D_0$, and this shows (by definition) that $(x_{M\setminus T},v) \in A(F(D_0))$.

For the second claim, suppose towards a contradiction that there exists $u \in N_{D_0}^+(v)$ such that $(u,x_{M\setminus T}) \in A(D_0)$. By definition of $D_0$, this means that $u \in N_D^+(v)$ and that there exists a vertex $m \in M\setminus T$ such that $(u,m) \in A(D)$. By definition of $T$, this however shows that $m \in T$, a contradiction. 
\end{proof}
In the following, let $D^\ast$ be the digraph defined by $$V(D^\ast):=V(D_0), A(D^\ast):=A(D_0) \cup \{(x_{M\setminus T},u)|u \in N_{D_0}^+(v)\}.$$
\paragraph{Claim 5.} $D^\ast \in \text{Forb}_\text{ind}(\bivec{K}_2,S_2^+,W_3^+)$.
\begin{proof}
Let $e_i=(x_{M\setminus T},u_i), i=1,\ldots,k$ be a list of the arcs contained in $A(D^\ast) \setminus A(D_0)$ for some $k \ge 0$. For $0 \le i \le k$ let $D_i$ denote the digraph defined by $V(D_i):=V(D_0)$ and $A(D_i):=A(D_0) \cup \{e_1,\ldots,e_i\}$. Note that $D_k=D^\ast$. 

We now claim that $D_i \in \text{Forb}_\text{ind}(\bivec{K}_2,S_2^+,W_3^+)$ and $(x_{M\setminus T},v) \in A(F(D_i))$ for every $i \in \{0,1,\ldots,k\}$ and prove this claim by induction on $i$.

For $i=0$ the claim holds true by the previous discussions and Claim 4. Now let $1 \le i \le k$ and suppose we know that the claim holds for $D_{i-1}$. 

Note that $D_i$ is the digraph obtained from $D_{i-1}$ by adding the arc $e_i=(x_{M\setminus T},u_i)$, where $u_i \in N_{D_0}^+(v)=N_{D_{i-1}}^+(v)$, $(x_{M\setminus T},v) \in A(F(D_{i-1}))$. Note that $e_i \notin A(D_{i-1})$, as well as $(u_i,x_{M\setminus T}) \notin A(D_{i-1})$ by Claim 4. Hence, we may apply Lemma~\ref{lem:arcaddition} to the digraph $D_{i-1}$ with $x=x_{M\setminus T}, y=v, z=u_i$ to find that indeed $D_i \in \text{Forb}_\text{ind}(\bivec{K}_2,S_2^+,W_3^+)$. It remains to show that $(x_{M\setminus T},v) \in A(F(D_i))$. However, the only new out-neighbor of $x_{M\setminus T}$ in $D_i$ compared to $D_{i-1}$ is the vertex $u_i$, which is still dominated by the vertex $v \in N_{D_i}^+(x_{M\setminus T})$ via the arc $(v,u_i)\in A(D_i)$, and hence $v$ still dominates all other out-neighbors of $x_{M\setminus T}$ in $D_i$. This shows that $D_i$ satisfies the induction claim. 

We have shown $D^\ast=D_k \in \text{Forb}_\text{ind}(\bivec{K}_2,S_2^+,W_3^+)$, concluding the proof of Claim~5.
\end{proof}

The number of vertices of $D^\ast$ satisfies $$v(D^\ast)=v(D_0)=v(D)-|T|-(|M\setminus T|-1) \le v(D)-(|M|-1) \le v(D)-2<v(D)$$ since $|M| \ge 3$ by Claim 3. Hence, the minimality of $D$ implies that the assertion of the theorem holds for $D^\ast$. Applying this assertion to the vertex $x_{M\setminus T}$ in $D^\ast$, we find that there exists an acyclic $2$-coloring $c^\ast:V(D^\ast) \rightarrow \{1,2\}$ of $D^\ast$ such that $c^\ast(x_{M\setminus T})=1=c^\ast(u)$ for every $u \in N_{D^\ast}^+(x_{M\setminus T})$. Using the facts $N_{D_0}^+(x_{M\setminus T}) \subseteq N_D^+(v) \cup \{v\}$, $N_{D_0}^+(v)=N_D^+(v)$ and $(x_{M\setminus T},v) \in A(D_0)$, the definition of $D^\ast$ yields that $N_{D^\ast}^+(x_{M\setminus T})=N_D^+(v) \cup \{v\}$. Hence, we have $c^\ast(x_{M\setminus T})=c^\ast(v)=c^\ast(u)=1$ for every $u \in N_D^+(v)$.

Let $c:V(D) \rightarrow \{1,2\}$ be the coloring of $D$ defined by $c(x):=c_M(x)$ for every $x \in M$, and $c(x):=c^\ast(x)$ for every $x \in V(D)\setminus M$. We note that $c(v)=c(u)$ for all $u \in N_D^+(v)$. Hence, by the initial assumption on $D$, the coloring $c$ cannot be acyclic, i.e., there is a directed cycle $C$ in $D$ which is monochromatic in the coloring $c$. Then we must have $V(C)\setminus M \neq \emptyset$, for otherwise $C$ would be a monochromatic directed cycle in the acyclic coloring $c_M$ of $D[M]$. Analogously, if $V(C) \cap M=\emptyset$, then $C$ would be a directed cycle in $D-M \subseteq (D-T)/(M\setminus T)=D_0 \subseteq D^\ast$, a contradiction. Therefore we also have $V(C) \cap M \neq \emptyset$, and hence there must be an arc $(x,y) \in A(C)$ such that $x \in M$ and $y \notin M$. However, this means that $y \in N_D^+(M) \subseteq \{v\} \cup N_D^+(v)$, and hence $c(y)=1$ by the above. Since $C$ is monochromatic, it follows that $V(C) \subseteq c^{-1}(1)$. In particular, since $c(t)=c_M(t)=2$ for every $t \in T$, it follows that $C$ is a directed cycle in $D-T$. Let $z$ be the first vertex of $M$ we meet when traversing the directed cycle $C$ in forward-direction, starting at $y$. Then $z \in M\setminus T$. Let $P$ be the unique directed $x,z$-path contained in $C$. Then $P$ has length at least two and satisfies $V(P) \cap M=\{x,z\}$ and $V(P)\subseteq c^{-1}(1)$. Now $(V(P)\setminus \{x,z\}) \cup \{x_{M\setminus T}\}$ forms the vertex-set of a directed cycle $C^\ast$ in $(D-T)/(M\setminus T)=D_0 \subseteq D^\ast$ containing $x_{M\setminus T}$, and we have $c^\ast(x)=c(x)=1$ for every vertex $x \in V(C^\ast)\setminus \{x_{M\setminus T}\}=V(P)\setminus \{x,z\} \subseteq V(D)\setminus M$. We have $c^\ast(x_{M\setminus T})=1$ by definition of $c^\ast$, and hence $C^\ast$ forms a monochromatic directed cycle of color $1$ in the acyclic coloring $c^\ast$ of $D^\ast$. This contradiction finally shows that our very first assumption, namely that a (smallest) counterexample $D$ to the claim of the theorem exists, was wrong. This concludes the proof of the theorem. 
\end{proof}

\section{$\{S_2^+, W_3^-\}$-Free Oriented Graphs}\label{sec:W3-}
In this section we prove Theorem~\ref{thm:W4minus}, showing that every digraph in $\text{Forb}_{\text{ind}}(\bivec{K}_2,S_2^+,W_3^-)$ is acyclically $4$-colorable. Note that $\text{Forb}_{\text{ind}}(\bivec{K}_2,S_2^+,W_3^-)$ is the class of oriented graphs $D$ such that the out-neighborhood of any vertex in $D$ spans a tournament, and the in-neighborhood of any vertex spans a directed triangle-free graph. In fact, we show the following strengthened statement.
\begin{theorem}
Let $D \in \text{Forb}_{\text{ind}}(\bivec{K}_2,S_2^+,W_3^-)$ and let $(u,v) \in A(D)$. Then $D$ admits an acyclic coloring $c:V(D) \rightarrow \{1,2,3,4\}$ satisfying the additional conditions
$c(u)=1$, $c(x)=1$ for all $x \in N_D^+(u)\setminus N_D^+(v)$ (so $c(v) = 1$) and $c(x) \in \{1,2\}$ for all $x \in N_D^+(v)$. 
\end{theorem}
\begin{proof}
Suppose towards a contradiction the claim of the theorem was wrong, and let $D$ be a counterexample minimizing $v(D)$. Then there exists an arc $(u,v) \in A(D)$ such that $D$ does \emph{not} admit an acyclic coloring $c:V(D) \rightarrow \{1,2,3,4\}$ satisfying the additional conditions
$c(u)=1$, $c(x)=1$ for all $x \in N_D^+(u)\setminus N_D^+(v)$ and $c(x) \in \{1,2\}$ for all $x \in N_D^+(v)$. Let us define $A:=N_D^+(u) \setminus (N_D^+(v) \cup \{v\})$ and $B:=N_D^-(u) \cap N_D^+(v)$. We start with some useful observations concerning these sets.
\paragraph{Claim 1.} $A \subseteq N_D^-(v)$, and $D[A]$ and $D[B]$ are transitive tournaments.
\begin{proof}
To show $A \subseteq N_D^-(v)$, let $x \in A=N_D^+(u) \setminus (N_D^+(v) \cup \{v\})$ be arbitrary. Since $(u,x),(u,v) \in A(D)$ and $x, u, v$ cannot induce an $S_2^+$ in $D$, the vertices $x$ and $v$ must be equal or adjacent in $D$. Since $x \notin N_D^+(v) \cup \{v\}$, it follows that $x \in N_D^-(v)$, as claimed.

Since $D[N_D^+(u)]$ is a tournament and $A \subseteq N_D^+(u)$, also $D[A]$ is a tournament. Furthermore $D[N_D^-(v)]$ is directed triangle-free, and with $A \subseteq N_D^-(v)$ also $D[A]$ is directed triangle-free, i.e., a transitive tournament, as claimed.

Similarly, since $D[N_D^-(u)]$ is directed triangle-free, and since $D[N_D^+(v)]$ is a tournament, $B=N_D^-(u) \cap N_D^+(v)$ implies that $D[B]$ must be both directed triangle-free and a tournament, i.e., a transitive tournament.
\end{proof}
In the following, let us denote by $D':=D-(N_D^-(u) \cup \{u\})$ the induced subdigraph of $D$ obtained by deleting the closed in-neighborhood of $u$. We clearly have $D' \in \text{Forb}_{\text{ind}}(\bivec{K}_2,S_2^+,W_3^-)$ and $v(D')<v(D)$, and hence by minimality of $D$ the theorem statement holds for $D'$. 
\paragraph{Claim 2.} There exists an acyclic coloring $c':V(D') \rightarrow \{1,2,3,4\}$ of $D'$ such that $c'(v)=1$, $c'(x)=1$ for all $x \in A$ and $c'(x) \in \{1,2\}$ for all $x \in N_{D'}^+(v)$.
\begin{proof}
We distinguish the two cases $A=\emptyset$ and $A \neq \emptyset$.

Suppose first that $A=\emptyset$. If $N_{D'}^+(v)=\emptyset$, then applying the theorem statement to $D'$ (for an arbitrarily chosen arc) yields that $\vec{\chi}(D') \le 4$, and hence there exists an acyclic coloring $c':V(D') \rightarrow \{1,2,3,4\}$. Since $v$ is a sink in $D'$, no directed cycle in $D'$ contains $v$. Consequently, we may assume w.l.o.g. (possibly by recoloring) that $c'(v)=1$. In particular, since $A=\emptyset, N_{D'}^+(v)=\emptyset$, the remaining two statements of Claim~2 are satisfied vacuously for $c'$, concluding the proof in this case. 

On the other hand, if $N_{D'}^+(v) \neq \emptyset$, then there exists an arc in $D'$ leaving $v$. Fix an arbitrary such arc $(v,y)$. Applying the Theorem statement to this arc in $D'$, we find that there is an acyclic coloring $c':V(D') \rightarrow \{1,2,3,4\}$ such that $c'(v)=1$, $c'(x) = 1$ for all $x \in N^+_{D'}(v) \setminus N^+_{D'}(y)$ and $c'(x) \in \{1,2\}$ for all $x \in N^+_{D'}(y)$; in particular, and $c'(x) \in \{1,2\}$ for all $x \in N_{D'}^+(v)$. Again, this shows that the claim holds true. 

Next suppose that $A \neq \emptyset$. By Claim 1, $D[A]$ is a transitive tournament. Let $a$ be the unique source-vertex of this tournament. Since $a \in A \subseteq N_D^-(v)$, it follows that $(a,v) \in A(D')$. Hence, we may apply the Theorem statement to the arc $(a,v)$ in $D'$ and find that there exists an acyclic coloring $c':V(D') \rightarrow \{1,2,3,4\}$ such that $c'(a)=1$, $c'(x)=1$ for all $x \in N_{D'}^+(a) \setminus N_{D'}^+(v)$ (in particular $c'(v)=1$), and $c'(x) \in \{1,2\}$ for all $x \in N_{D'}^+(v)$.
Since $a$ is the source of $D[A]=D'[A]$ and since $A \cap N_D^+(v)=\emptyset$, we have $\{a\} \cup (N_{D'}^+(a) \setminus N_{D'}^+(v)) \supseteq A$ and thus $c'(x)=1$ for all $x \in A$, as required. This shows the assertion of Claim 2 and concludes the proof.
\end{proof}
\paragraph{Claim 3.} There exists an acyclic coloring $c^-:N_D^-(u) \rightarrow \{2,3,4\}$ of $D[N_D^-(u)]$ such that $c^-(x)=2$ for all $x \in B$ and $c^-(x) \in \{3,4\}$ for all $x \in N_D^-(u)\setminus B$.
\begin{proof}
Since $D \in \text{Forb}_{\text{ind}}(\bivec{K}_2,S_2^+,W_3^-)$, we have $D[N_D^-(u)] \in \text{Forb}_{\text{ind}}(\bivec{K}_2,S_2^+,\vec{C}_3)$. By Theorem~\ref{thm:dirtriangle} there exists an acyclic coloring of $D[N_D^-(u)]-B$ using only colors $3$ and $4$. Clearly, $D[B]$ as a transitive tournament (see~Claim 1) admits an acyclic coloring only with color $2$. Putting these colorings together yields an acyclic coloring $c'$ of $D[N_D^-(u)]$ with the required properties.
\end{proof}
Let $c:V(D) \rightarrow \{1,2,3,4\}$ be the coloring defined by $c(u):=1$, $c(x):=c^-(x)$ for all $x \in N_D^-(u)$ and $c(x):=c'(x)$ for all $x \in V(D)\setminus (N_D^-(u) \cup \{u\})$. 

Note that from the properties of $c'$ given by Claim 2 we have $c(u)=c(v)=1$ and $c(x)=1$ for all $x \in A=N_D^+(u)\setminus (N_D^+(v) \cup \{v\})$. Furthermore, since $N_D^+(v)=N_{D'}^+(v) \cup B$, the properties of $c'$ and $c^-$ imply that $c(x) \in \{1,2\}$ for all $x \in N_D^+(v)$. 

Given these properties, our initial assumption concerning $D$ implies that $c$ cannot be an acyclic coloring of $D$, that is, there is a directed cycle $C$ in $D$ which is monochromatic under $c$. Since $c'=c|_{V(D')}$ and $c^-=c|_{N_D^-(u)}$, we must have $V(C) \cap (N_D^-(u) \cup \{u\}) \neq \emptyset$ and $V(C) \setminus N_D^-(u) \neq \emptyset$, for otherwise $c'$ resp.~$c^-$ would not be acyclic. Further note that $u \notin V(C)$, for every arc in $D$ entering $u$ has its tail colored with either $2, 3$ or $4$, while its head, $u$, receives color $1$ under $c$ (so a directed cycle containing $u$ cannot be monochromatic). Hence, there must exist an arc $(x,y) \in A(C)$ such that $x \in N_D^-(u)$ and $y \in V(D)\setminus (N_D^-(u) \cup \{u\})$. Since $(x,u), (x,y) \in A(D)$ and $D$ is induced $S_2^+$-free, $u$ and $y$ must be equal or adjacent, and since $y \notin N_D^-(u) \cup \{u\}$, we have $y \in N_D^+(u)$. By the properties of $c'$ and $c^-$, we have $N_D^+(u)\setminus N_D^+(v)=A \cup \{v\} \subseteq c^{-1}(\{1\})$, $B \subseteq c^{-1}(\{2\})$ and $N_D^-(u)\setminus B \subseteq c^{-1}(\{3,4\})$. The cycle $C$ is monochromatic, therefore $c(x)=c(y)$. From this we conclude that $y \in N_D^+(v)$, and hence $c(y) \in \{1,2\}$. This is only possible if $c(x)=c(y)=2$, and hence $x \in B$. It follows that $x,y,u,v \in V(D)$ are distinct vertices satisfying $(x,y),(u,y),(v,y) \in A(D)$, as well as $(x,u), (u,v), (v,x) \in A(D)$ (here we used that $x \in B=N_D^-(u) \cap N_D^+(v)$). This however means that $x,y,u,v$ induce a copy of $W_3^-$ in $D$, which is absurd considering that $D \in \text{Forb}_{\text{ind}}(\bivec{K}_2,S_2^+,W_3^-)$. This shows that our very first assumption concerning the existence of a smallest counterexample $D$ was wrong. This concludes the proof of the theorem.
\end{proof}

\section{Adding a Dominating Sink to a Hero}\label{sec:addsink}

In this section our goal is to prove Theorem~\ref{thm:addsink}. Let us first prove the following lemma.

\begin{lemma}\label{lemma:shortpath}
Let $D \in \text{Forb}_\text{ind}(\bivec{K}_2, S_2^+)$ and let $C \in \mathbb{N}$ be such that $\vec{\chi}(D[N_D^-(x)]) \le C$ for every $x \in V(D)$. Let $u, v \in V(D)$ and let $P$ be a shortest $u$-$v$-dipath in $D$. Let $X:=V(P) \cup N_D^-(V(P))$. Then $\vec{\chi}(D[X]) \le 3C+2$.
\end{lemma}
\begin{proof}
Let $u=x_0,x_1, \ldots,x_{\ell-1},x_\ell=v$ be the vertex-trace of $P$ and consider the partition $(A_i)_{i=1}^{\ell}$ of $N_D^-(V(P))$ where $A_i:=N_D^-(x_i)\setminus (V(P) \cup \bigcup_{1 \le j<i}{A_j}), i=0,\ldots,\ell$. 
\paragraph{Claim.} Let $0 \le i<j \le \ell$ with $j-i \ge 3$. Then there exists no arc in $D$ starting in $A_i$ and ending in $A_j$.

\begin{proof}
Suppose towards a contradiction that there are vertices $x \in A_i$, $y \in A_j$ with $(x,y) \in A(D)$. Since $(x,x_i) \in A(D), (x,y) \in A(D)$ and $x_i \neq y$ (since $x_i \in V(P)$ and $y \notin V(P)$), $x_i$ and $y$ must be adjacent in $D$. By definition of $A_j$ we have $A_j \cap N_D^-(x_i)=\emptyset$ and hence $(x_i,y) \in A(D)$. However, now the directed path described by the vertices $u=x_0,x_1,\ldots,x_i,y,x_j,\ldots,x_\ell=v$ is a $u$-$v$-dipath in $D$ shorter than $P$, a contradiction. This proves the claim.
\end{proof}
For every $0 \le i \le \ell$ we have $\vec{\chi}(D[A_i]) \le \vec{\chi}(D[N_D^-(x_i)]) \le C$. Let us define the set $B_r:=\bigcup\{A_i|i \equiv r \text{ (mod }3)\}$ for every $r \in \{0,1,2\}$. From the above claim it follows that no directed cycle in $D[B_r]$ intersects two different sets $A_i, A_j$. Hence, we have $$\vec{\chi}(D[B_r]) \le \max\{\vec{\chi}(D[A_i])|i \equiv r \text{ (mod }3)\} \le C$$ for $r=0,1,2$. Further note that the two sets $$V_0:=\{x_i|i \in \{0,\ldots,\ell\} \text{ even}\}, V_1:=\{x_i|i \in \{0,\ldots,\ell\} \text{ odd}\}$$ both induce acyclic subdigraphs of $D$, for otherwise $D$ would not be a shortest $u$-$v$-dipath in $D$. Since $X$ is the disjoint union of $B_0, B_1, B_2, V_0, V_1$, we conclude 
$$\vec{\chi}(D[X]) \le \vec{\chi}(D[B_0])+\vec{\chi}(D[B_1])+\vec{\chi}(D[B_2])+\vec{\chi}(D[V_0])+\vec{\chi}(D[V_1]) \le 3C+2,$$ as required.
\end{proof}

\begin{proof}[Proof of Theorem~\ref{thm:addsink}]
Let $\{\bivec{K}_2, S_2^+, H\}$ be heroic and $C:=\vec{\chi}(\text{Forb}_\text{ind}(\bivec{K}_2, S_2^+, H))$. 

We claim that every digraph $D \in \text{Forb}_\text{ind}(\bivec{K}_2, S_2^+, H^-)$ admits an acyclic coloring with $C^-:=v(H)(C+1)+3C+2$ colors. 

Suppose towards a contradiction that there exists some $D \in \text{Forb}_\text{ind}(\bivec{K}_2, S_2^+, H^-)$ with $\vec{\chi}(D')>C'$, and choose such a $D$ minimizing $v(D)$. Then we have $\vec{\chi}(D) > C' \ge C$ and hence there is $Y \subseteq V(D)$ such that $D[Y]$ is isomorphic to $H$. Furthermore, since the dichromatic number of $D$ is the maximum of the dichromatic numbers of its strong components, the minimality of $v(D)$ implies that $D$ is strongly connected.

Let $S \supseteq Y$ denote a set of vertices in $D$ defined as follows: 

If $D[Y]$ (resp.~$H$) is strongly connected, put $S:=Y$. Otherwise, let $Y_1,\ldots,Y_t$ be a partition of $Y$ into the $t \ge 2$ strong components of $D[Y]$ such that all arcs between $Y_i$ and $Y_j$ start in $Y_i$ and end in $Y_j$, for any $1 \le i <j \le t$ (note that since $D[Y]$ is a tournament all elements of $Y_i \times Y_j$ are arcs of $D[Y]$ for $1 \le i<j \le t$). Now pick $u \in Y_t,$ $v \in Y_1$ arbitrarily, let $P$ be a shortest $u$-$v$-dipath in $D$ and put $S:=V(P) \cup Y$. Let us note that in any case, $D[S]$ is strongly connected.

Let $Z:=S \cup N_D^-(S)$. Then we have $Z=X \cup Y \cup N_D^-(Y)$, where $X$ is defined as $X:=\emptyset$ if $S=Y$ and as $X:=V(P) \cup N_D^-(V(P))$ otherwise. For every $x \in V(D)$ we know that since $D$ is $H^-$-free, the digraph $D[N_D^-(x)]$ is contained in $\text{Forb}_\text{ind}(\bivec{K}_2, S_2^+, H)$, and hence $\vec{\chi}(D[N_D^-(x)]) \le C$. Using Lemma~\ref{lemma:shortpath} we obtain that $\vec{\chi}(D[X]) \le 3C+2$. Putting it all together, we find that
$$\vec{\chi}(D[Z]) \le \sum_{y \in Y}{\underbrace{\vec{\chi}(D[\{y\} \cup N_D^-(y)])}_{\le C+1}}+\vec{\chi}(D[X]) \le v(H)(C+1)+3C+2=C'.$$ 
\paragraph{Claim.} No arc in $D$ leaves $Z$. 
\begin{proof}
We first show that there do not exist $z \in S$, $w \in V(D)\setminus Z$ such that $(z,w) \in A(D)$. Suppose towards a contradiction that such an arc $(z,w)$ exists. We claim that then $(s,w) \in A(D)$ for every $s \in S$. Consider $s \in S$ arbitrarily. Since $D[S]$ is strongly connected, there exist vertices $z=s_0,s_1,\ldots,s_k=s$ in $S$ such that $(s_{i-1},s_i) \in A(D)$, $i=1,\ldots,k$. We show $(s_i,w) \in A(D)$ for all $i=0,\ldots,k$ by induction on $i$. Clearly it is true for $i=0$, so suppose that $1 \le i \le k$ and we have established that $(s_{i-1},w) \in A(D)$. Since $w \notin Z$, $s_i \in Z$, we have $w \neq s_i$ and $(s_{i-1},w), (s_{i-1},s_i) \in A(D)$. Since $D$ is $S_2^+$-free, it follows that $s_i$ and $w$ are adjacent. However, since $w \notin Z=S \cup N_D^-(S) \supseteq N_D^-(s_i)$, we must have $(s_i,w) \in A(D)$, as claimed. 

This shows that indeed $(s,w) \in A(D)$ for all $s \in S$. However, since $S \supseteq Y$ and since $D[Y]$ is isomorphic to $H$, it follows that $D[Y \cup \{w\}]$ is an induced subdigraph of $D$ isomorphic to $H^-$, a contradiction to $D\in \text{Forb}_\text{ind}(\bivec{K}_2, S_2^+, H^-)$. This shows that there are no arcs from $S$ to $V(D)\setminus Z$.

To complete the proof, let us show that there are no arcs starting in $Z \setminus S=N_D^-(S)$ that end in $V(D)\setminus Z$. Suppose towards a contradiction that there exist $z \in N_D^-(S)$ and $w \in V(D)\setminus Z$ with $(z,w) \in A(D)$. Then there is a vertex $s \in S$ such that $(z,s) \in A(D)$. Since $s \neq w$, $(z,s),(z,w) \in A(D)$ and $D$ is $S_2^+$-free we find that $s$ and $w$ are adjacent. Since $w \notin Z \supseteq N_D^-(s)$, it follows that $(s,w) \in A(D)$. However, this yields a contradiction, since we showed above that no arc in $D$ starts in $S$ and ends in $V(D)\setminus Z$. 
All in all, the claim follows.
\end{proof} 
Since $D$ is strongly connected and $Z \neq \emptyset$ (since $Z \supseteq Y$), it follows that $Z=V(D)$, and hence that $\vec{\chi}(D)=\vec{\chi}(D[Z]) \le C'$, a contradiction to our initial assumption. This concludes the proof of the theorem.
\end{proof}

\section{Oriented $4$-Vertex-Paths}\label{sec:4path}
In this section we establish that $\{\bivec{K}_2,\vec{K}_3,P^+(1,1,1)\}$ is heroic proving Theorem~\ref{thm:N}.
\begin{proof}[Proof of Theorem~\ref{thm:N}]
We prove by induction on $n$ that every directed graph on $n$ vertices $D \in \text{Forb}_\text{ind}(\bivec{K}_2,\vec{K}_3,P^+(1,1,1))$ admits an acyclic $2$-coloring. The claim trivially holds for $n=1$, so let $n \ge 2$ and suppose that every digraph in $\text{Forb}_\text{ind}(\bivec{K}_2,\vec{K}_3,P^+(1,1,1))$ having less than $n$ vertices is $2$-colorable. Pick some $x \in V(D)$ arbitrarily. Let us define a sequence $X_0,X_1,X_2,\ldots$ of subsets of $V(D)$ as follows: 

$$X_i:=\begin{cases} \{x\}, & \text{ if }i=0, \cr N^+(X_{i-1}) \setminus \bigcup_{j=0}^{i-1}{X_j}, & \text{ if }i \text{ odd}, \cr N^-(X_{i-1}) \setminus \bigcup_{j=0}^{i-1}{X_j}, & \text{ if }i \ge 2 \text{ even}. \end{cases}$$
The sets $(X_i)_{i \ge 0}$ are by definition pairwise disjoint, and so there exists $k \ge 1$ such that $X_1,\ldots,X_k \neq \emptyset$ and $X_i=\emptyset$ for all $i>k$. 
\paragraph{Claim.} $X_i$ is an independent set of $D$ for every $i \ge 0$.
\begin{proof}
We prove the claim by induction on $i$. The claim trivially holds for $i=0$ since $X_0=\{x\}$, and since $D$ does not contain a transitive triangle, also $X_1=N^+(x)$ must be an independent set in $D$. Now let $i \ge 2$ and suppose that we already established that $X_0,\ldots,X_{i-1}$ are independent. To show that $X_i$ is independent, let us suppose towards a contradiction that there are $x,y \in X_i$ such that $(x,y) \in A(D)$. By definition of the sets $X_i$ there are vertices $x_1,y_1 \in X_{i-1}$ and $x_2,y_2 \in X_{i-2}$ such that the following holds: $(x_1,x_2),(x_1,x),(y_1,y_2), (y_1,y) \in A(D)$ if $i$ is odd, respectively $(x_2,x_1),$ $(x,x_1),$ $(y_2,y_1),$ $(y,y_1) \in A(D)$ if $i$ is even. We must have $x_1 \neq y_1$ in any case, since otherwise the vertices $x_1=y_1,x,y$ would induce a $\vec{K}_3$ in $D$. Let us now consider the oriented $4$-vertex-path $P$ in $D$ defined as $P=x,(x,y),y,(y_1,y),y_1,(y_1,y_2),y_2$ if $i$ is odd, respectively as $P=x_2,(x_2,x_1),x_1,(x,x_1),x,(x,y),y$ if $i$ is even. In order for this path not to be an induced copy of $P^+(1,1,1)$, two non-consecutive vertices of the path must be adjacent. However, since $D$ does not contain transitive triangles, this is only possible if $x$~and~$y_2$ ($i$ odd) respectively $x_2$ and $y$ ($i$ even) are adjacent. Since $x \notin X_{i-1}$, we have $x \notin N^-(X_{i-2})$ if $i$ is odd and $y \notin N^+(X_{i-2})$ if $i$ is even. Since $x_2,y_2 \in X_{i-2}$ we conclude that $(y_2,x) \in A(D)$ if $i$ is odd and $(y,x_2) \in A(D)$ if $i$ is even. In both cases we conclude that $x_2 \neq y_2$, since otherwise the vertices $x_2=y_2,x_1,x$ respectively $x_2=y_2,y_1,y$ would induce a transitive triangle in $D$. Now consider the oriented path $Q$ in $D$ defined as $Q=y_2,(y_2,x),x,(x_1,x),x_1,(x_1,x_2),x_2$ if $i$ is odd and as $Q=y_2,(y_2,y_1),y_1,(y,y_1),y,(y,x_2),x_2$ if $i$ is even. In order for $Q$ not to be an induced copy of $P^+(1,1,1)$ in $D$ and since $D$ does not contain transitive triangles, this implies in both cases that the endpoints $x_2$ and $y_2$ of $Q$ must be adjacent. This contradicts the induction hypothesis that $X_{i-2}$ is an independent set. Hence, our assumption was wrong, $X_i$ is indeed independent. This concludes the proof of the claim.
\end{proof}

Let $X:=X_0 \cup \dots \cup X_k$ and $D':=D-X$. By the induction hypothesis $D'$ admits an acyclic $2$-coloring $c':V(D') \rightarrow \{1,2\}$. Let us now define $c:V(D) \rightarrow \{1,2\}$ by $c(x):=c'(x)$ for every $x \in V(D) \setminus X$, $c(x):=1$ for every $x \in X_i$ such that $i$ is even, and $c(x):=2$ for every $x \in X_i$ such that $i$ is odd. We claim that $D$ defines an acyclic coloring of $D$. Suppose towards a contradiction that there exists a monochromatic directed cycle $C$ in $(D,c)$. Since $c'$ is an acyclic coloring of $D'$, we must have $V(C) \cap (X_0 \cup \dots \cup X_k) \neq \emptyset$. Note that by definition of the sets $(X_i)_{i \ge 0}$ we have $N^+\left(\bigcup_{i \text{ even}}X_i\right), N^-\left(\bigcup_{i \text{ odd}}X_i\right) \subseteq X$. Hence, no arc of $D$ starts in $c^{-1}(\{1\}) \cap X$ and ends in $V(D)\setminus X$, and no arc of $D$ starts in $c^{-1}(\{2\}) \cap X$ and ends in $V(D)\setminus X$. Since $V(C)\subseteq c^{-1}(t)$ for some $t \in \{1,2\}$, the strong connectivity of $C$ shows that in fact $V(C) \subseteq c^{-1}(t) \cap X$ for some $t \in \{1,2\}$. Let $i_0 \ge 0$ be smallest such that $X_{i_0} \cap V(C) \neq \emptyset$. Let $u \in X_{i_0} \cap V(C) \neq \emptyset$, and let $u^-,u^+ \in V(C)$ be such that $(u^-,u),(u,u^+) \in A(C)$. Since $X_{i_0}$ is an independent set, and by definition of the coloring $c$, we must have $u^- \in X_{i_0+2s^-},u^+ \in X_{i_0+2s^+}$ for integers $s^-,s^+ \ge 1$. On the other hand, we have $u^+ \in N^+(X_{i_0}) \setminus \bigcup_{j=0}^{i_0-1}{X_j}=X_{i_0+1}$ if $i_0$ is even and $u^- \in N^-(X_{i_0}) \setminus \bigcup_{j=0}^{i_0-1}{X_j}=X_{i_0+1}$ if $i_0$ is odd, in both cases yielding a contradiction since $X_{i_0+2s^+},X_{i_0+2s^-}$ are disjoint from $X_{i_0+1}$. This shows that our assumption was wrong, indeed, $c$ is an acyclic coloring of $D$. Hence, $\vec{\chi}(D) \le 2$, concluding the proof.
\end{proof}

\section{Conclusion}\label{sec:inducedconclusion}
In the first three sections of this paper we have proved that set $\{\bivec{K}_2,S_2^+,H\}$ is heroic for several small heroes $H$, and in particular we resolved Conjecture~\ref{con:S2+}. It would be interesting to prove that in fact, for any hero $H$, $\{\bivec{K}_2,S_2^+,H\}$ is heroic, as this would be a broad generalization of the main result of Berger et al.~\cite{heros} from tournaments to \emph{locally out-complete} oriented graphs, i.e., oriented graphs in which the out-neighborhood of every vertex induces a tournament. This class of digraphs has been thoroughly studied in the past, see for instance~\cite{bangjensen2} for a survey of results on locally complete digraphs. 

The smallest open case of this problem would be to show that $\{\bivec{K}_2,S_2^+,\vec{K}_4^s\}$ is heroic, where $\vec{K}_4^s$ denotes the unique strong tournament on four vertices. It seems that already for this case a new method is required. We do however believe that the following is true.
\begin{conjecture} 
$\vec{\chi}(\text{Forb}_\text{ind}(\bivec{K}_2,S_2^+,\vec{K}_4^s))=3$.
\end{conjecture}
Here, a tight lower bound would be provided by the following construction: Take a $3$-fold blow-up of a directed $4$-cycle (every arc being replaced by an oriented $K_{3,3}$) and connect each of the three blow-up triples by a directed triangle. This oriented graph is contained in $\text{Forb}_\text{ind}(\bivec{K}_2,S_2^+,\vec{K}_4^s)$ and has dichromatic number $3$.

Let us further remark at this point that there exists a very simple proof that if we exclude \emph{both} $S_2^+$ and $S_2^-$, i.e., we consider \emph{locally complete} oriented graphs (where the in- \emph{and} out-neigborhood of every vertex induces a tournament), then we can show that the exclusion of any hero indeed bounds the dichromatic number as follows.
\begin{remark} 
For any hero $H$, we have $$\vec{\chi}(\text{Forb}_\text{ind}(\bivec{K}_2,S_2^+, S_2^-,H)) \le 2\vec{\chi}(\text{Forb}_\text{ind}(\bivec{K}_2,\overline{K}_2,H))<\infty.$$
\end{remark}
\begin{proof}
By the result of Berger et al.~\cite{heros} we have $C:=\vec{\chi}(\text{Forb}_\text{ind}(\bivec{K}_2,\overline{K}_2,H))<\infty$. Let us now prove that  $\vec{\chi}(\text{Forb}_\text{ind}(\bivec{K}_2,S_2^+, S_2^-,H)) \le 2C$. Towards a contradiction suppose that $\vec{\chi}(D)>2C$ for some $D \in\text{Forb}_\text{ind}(\bivec{K}_2,S_2^+, S_2^-,H)$, and pick $D$ such that $v(D)$ is minimum. Pick $v \in V(D)$ arbitrarily and define $D':=D-(\{v\} \cup N_D(v))$. 

Since $v(D')<v(D)$, there exists an acyclic $2C$-coloring $c':V(D') \rightarrow \{1,\ldots,2C\}$ of $D'$. Since $D$ is induced $S_2^+, S_2^-$-free, we further know that $D^+:=D[\{v\} \cup N_D^+(v)]$ and $D^-:=D[N_D^-(v)]$ are tournaments excluding $H$. It follows from the definition of $C$ that there exists an acyclic $C$-coloring $c^+:V(D^+) \rightarrow \{1,\ldots, C\}$ of $D^+$ as well as an acyclic coloring $c^-:V(D^-) \rightarrow \{C+1,\ldots,2C\}$ of $D^-$. Let $c$ be the $2C$-coloring of $D$ defined as the common extension of $c',c^+,c^-$ to $V(D)$. Since $\vec{\chi}(D)>2C$ there exists a directed cycle $C$ which is monochromatic under $c$. Since $c', c^+, c^-$ are acyclic colorings and since the color sets used by $c^+$ and $c^-$ are disjoint, we must have $V(C) \cap (\{v\} \cup N_D(v)) \neq \emptyset$, $V(C) \setminus (\{v\} \cup N_D(v)) \neq \emptyset$. Since all in-neighbors of $v$ have a distinct color from $v$, we further have $v \notin V(C)$. We conclude that there are vertices $x_1,x_2 \in V(C) \cap N_D(v)$, $y_1,y_2 \in V(C)\setminus (\{v\} \cup N_D(v))$ such that $(x_1,y_1), (y_2,x_2) \in A(C)$. We claim that we must have $x_1 \in N_D^+(v)$ and $x_2 \in N_D^-(v)$. Indeed, otherwise we would have $(x_1,v) \in A(D)$ or $(v,x_2) \in A(D)$, and then either the vertices $x_1,v,y_1$ induce an $S_2^+$ in $D$, or $x_2,v,y_2$ induce an $S_2^-$ in $D$, in each case yielding a contradiction to $D \in \text{Forb}_\text{ind}(\bivec{K}_2,S_2^+, S_2^-,H)$. Finally, we conclude that $c(x_1)=c^+(x_1) \le C<c^-(x_2)=c(x_2)$, a contradiction to the facts that $C$ is monochromatic and $x_1,x_2 \in V(C)$. This shows that our initial assumption concerning the existence of $D$ was wrong, concluding the proof of the remark.
\end{proof}

In the last section of this paper we investigated oriented graphs excluding the anti-directed $4$-vertex-path $P^+(1,1,1)$. It would certainly be very interesting and insightful to generalize both Theorem~\ref{thm:N} as well as the result of Aboulker et al.~concerning $\vec{P}_4$ by proving that $\{\bivec{K}_2,\vec{P}_4,\vec{K}_k\}$ and $\{\bivec{K}_2,P^+(1,1,1),\vec{K}_k\}$ are heroic for all $k \ge 4$.
\paragraph*{Acknowledgements.} My sincerest thanks go to Lior Gishboliner and Tibor Szab\'{o} for stimulating and fruitful discussions on the topic which contributed to the results presented in this paper. I would also like to thank Lior Gishboliner for improving the presentation of some parts of the write-up.

\end{document}